\documentclass[a4paper,11pt]{amsart}
\usepackage{standalone}
\usepackage{amsfonts}
\usepackage{amsthm}
\usepackage{amssymb}
\usepackage{amsmath}
\usepackage{theoremref}
\usepackage{enumerate}
\usepackage{bbm}
\usepackage{bm}
\usepackage{pgfplots}
\pgfplotsset{compat=1.18} 
\usepgfplotslibrary{fillbetween}
\usetikzlibrary{intersections}
\usepackage{mathtools}
\usetikzlibrary{patterns}

\numberwithin{equation}{section}

\newcommand{\C}{\mathcal{C}}

\newcommand{\LL}{\mathcal{L}}

\newcommand{\T}{\mathbb{T}}

\newcommand{\D}{\mathbb{D}}
\newcommand{\DD}{\mathcal{D}}
\newcommand{\FF}{\mathcal{F}}
\newcommand{\B}{\mathcal{B}}
\newcommand{\R}{\mathbb{R}}

\newcommand{\Po}{\mathcal{P}}

\newcommand{\hil}{\mathcal{H}}

\newtheorem{thm}{Theorem}[section]
\newtheorem*{thm*}{Theorem}
\newtheorem{lem}[thm]{Lemma}

\newtheorem{cor}[thm]{Corollary}
\newtheorem*{cor*}{Corollary}

\theoremstyle{definition}

\theoremstyle{definition}

\newtheorem{defn}[thm]{Definition}
\newtheorem*{defn*}{Definition}

\newtheorem{claim*}{Claim}

\title[Distributing mass under a pointwise bound and an application]{Distributing mass under a pointwise bound and an application to weighted polynomial approximation}
\author{Linus Bergqvist and Bartosz Malman}

\address{Division of Mathematics and Physics, 
        Mälardalen University,
		721 23 Västerås, Sweden}
\email{bartosz.malman@mdu.se} 

\address{Stockholm, Sweden}
\email{linus.lidman.bergqvist@gmail.com} 

\begin{document}

\begin{abstract}
    Inspired by applications in weighted polynomial approximation problems, we study an optimal mass distribution problem. Given a gauge function $h$ and a non-negative ``roof" function $R$ compactly supported in $\R^n$, we are interested in estimating the supremum of the $L^1$-norms of non-negative functions $f$ satisfying the pointwise bound $f \leq R$ and the mass distribution bound $\int_c f \, dV_n \leq h(V_n(c))$, where $c$ is a cube and $V_n$ is the volume measure. We prove a duality theorem which states that the optimal value in this maximization problem is the minimum among certain quantities associated with semi-covers by cubes of the support of $R$. We use our theorem to solve the so-called \textit{splitting problem} in the theory of polynomial approximations in the plane. As a result, we confirm an old conjecture of Kriete and MacCluer regarding an extension of Khrushchev's original splitting theorem to the weighted context, and explain the mechanics of an example in the research problem book by Havin, Khrushchev and Nikolskii.
\end{abstract}

\thanks{Part of this work has been done during a visit of the second author at Département de mathématiques et de statistique, Université Laval, supported by Simons-CRM Scholar-in-Residence program. He gratefully acknowledges the financial support, and thanks the department for the excellent research environment provided.}

\maketitle

\section{Introduction}

\subsection{A classical mass distribution result} We start by recalling a basic principle in geometric measure theory, which sometimes goes under the name of \textit{Frostman's lemma}, the proof of which can be found in numerous textbooks on the subject. Given a compact set $E$ in the $n$-dimensional Euclidean space $\R^n$ equipped with the volume measure $V_n$, and a real number $\beta \in (0, n)$, one is interested in the maximal amount of mass $\mu(E)$ one can pack on the set $E$ under an upper bound condition on local mass distribution of the form $\mu(c) \leq V_n(c)^{\beta/n}$, where $c$ is a cube in $\R^n$ and $V_n(c)$ is the volume of $c$. Here, of course, $\mu$ is a Borel measure supported on $E$, and a cube $c$ is a subset of $\R^n$ of the form $c = I_1 \times \ldots \times I_n$, where $I_j$ are closed intervals in $\R$ with a common length.

The various formulations of Frostman's lemma postulate that a non-trivial amount of mass $\mu(E) > 0$ can be placed on $E$ under the above distribution condition if and only if the Hausdorff measure $\hil_\beta(E)$ is positive.  Additionally, one is often interested in how large the positive quantity $\mu(E)$ can be. A quantitative version of the result states that
\begin{equation} \label{E:FrostmanDuality}
    \sup_\mu \mu(E) \simeq \inf_{\{c_k\}_k} \sum_{k}  V_n(c_k)^{\beta/n},
\end{equation} where the supremum on the left-hand side is taken over all positive Borel measures supported on $E$ and satisfying the bound $\mu(c) \leq V_n(c)^{\beta/n}$, and the infimum on the right-hand side is taken over all the covers of $E$ by cubes $c_k$. The implied constant depends on the dimension $n$. Proofs of the above statements can be found in \cite{carleson1967selected}, \cite{mattila1999geometry}, and various other sources.

Frostman's lemma is undoubtedly a valuable tool in geometric measure theory and beyond. In this article, we shall discuss a variant of Frostman's lemma which deals with positive functions $f$ instead of measures $\mu$, where the support condition is replaced by a pointwise bound on $f$. This change of setting is motivated by applications to certain outstanding problems in the theory of weighted polynomial approximation in the plane, but our main theorem (\thref{T:Maindualitytheorem} below) is hopefully of independent interest.

\subsection{Primal problem} \label{S:PrimalSubsec}We need to introduce a few more definitions in order to state our main theorem in a generality necessary for our intended applications.

We shall say that $h$ is a \textit{gauge function} if it is a continuous and increasing function defined for non-negative real numbers, and it satisfies $h(0) = 0$. We shall require our gauge functions $h$ to also satisfy the regularity conditions 

\begin{equation} \label{E:Reg1}
h(x)/x \text{ is decreasing in } x \tag{R1}  \end{equation}
and 
\begin{equation} \label{E:Reg2}
    \lim_{x \to 0} \frac{h(x)}{x} = \infty. \tag{R2}  
\end{equation}
Note that $\eqref{E:Reg1}$ implies the subadditivity property $h(x+y) \leq h(x) + h(y)$.

Important examples are \[h(x) = x \log(1/x) \quad \text{ (for small } x) \]and
\[h(x)  =  x^\alpha, \quad \alpha \in (0,1)\] 


A \textit{roof function} $R$ is a compactly supported non-negative measurable function from $\R^n$ into $[0, \infty]$. We make no assumptions regarding integrability of $R$ with respect to $V_n$, and we emphasize that we allow $R$ to attain the value $\infty$ on sets of positive volume.

\begin{defn} \thlabel{D:FclassDef}
 Given a positive integer $n$, a roof function $R$, and a gauge function $h$, we define the class $\FF_n(R, h)$ of all measurable functions $f: \R^n \to [0, \infty)$ which satisfy the following two conditions:
\begin{enumerate}[(i)]
    \item $\int\displaylimits_c f \, dV_n \leq h(V_n(c))$ for every cube $c \subset \R^n$,
    \item $f \leq R$ almost everywhere in $\R^n$ with respect to the volume measure $V_n$.
\end{enumerate}
\end{defn}

The following is our primal problem. In the family $\FF_n(R,h)$ we seek to find $f$ for which the quantity \[ \|f\|_1 := \int\displaylimits_{\R^n} f \, dV_n\] is as large as possible. 

The problem has an obvious similarity to the above described classical problem of packing maximal amount of mass $\mu(E)$ on a compact set $E$ under the condition that the distribution bound $\mu(c) \leq h(V_n(c))$ is satisfied for cubes. The parameter of our optimization problem is a function instead of a general measure, and the support condition is replaced by the weaker pointwise condition $f \leq R$. Note that we may pass back to a support condition by introducing a function $R_E$ for which we have $R_E = \infty$ on $E$ and $0$ elsewhere.

\subsection{Dual problem} Let $R$ be as before, and consider the quantity 
\begin{equation} \label{E:GenHausdorffContDef}
M_h(R) := \inf_{ \{ c_k\}_k} \, \sum_k h(V_n(c_k)) + \int\displaylimits_{\R^n \setminus \cup_{k} c_k} R \, dV_n
\end{equation} where the infimum on the right-hand side is taken over all countable families of cubes $\{c_k\}_k$. Here we essentially attempt to minimize the $L^1$-integral of $R$ with the twist that we can locally replace the contribution of $R$ on a cube $c$ by the quantity $h(V_n(c))$.

Roughly speaking, our mapping $R \mapsto M_h(R)$ is a functional version of a Hausdorff-type measure. Namely, if $E \subset \R^n$ is a set of positive volume measure, and as above, $R_E$ is the function supported on $E$ taking on the value $\infty$ on $E$ and $0$ on $\mathbb{R}^n \setminus E$, then $M_h(R_E)$ in \eqref{E:GenHausdorffContDef} is similar to the quantity appearing on the right-hand side of \eqref{E:FrostmanDuality}, but the new version is instead equal to the infimum of the values $\sum_k h(V_n(c_k))$, where $\{c_k\}_k$ is any \textit{up-to-measure-zero} cover of $E$. Such covers have found applications in polynomial approximation problems studied in \cite{malman2023thomson}.

\subsection{Duality theorem}

The two criteria on $f$ for membership in the family $\FF_n(R,h)$ put some restrictions on the size of the quantity $\|f\|_1$. For instance, if $\{c_k\}_k$ is any countable set of cubes, then we may use properties $(i)$ and $(ii)$ above to estimate
\begin{align*}
    \| f\|_{1} &=  \Big(\sum_{k} \int\displaylimits_{c_k} f\, dV_n \Big) +  \int\displaylimits_{\R^n \setminus \cup_{k} c_k} f \, dV_n \\
    &\leq \sum_k h(V_n(c_k)) + \int\displaylimits_{\R^n \setminus \cup_{k} c_k} R \, dV_n.
\end{align*}
Thus by definition \eqref{E:GenHausdorffContDef}, it follows that \[ \sup_{f \in \FF_n(R,h)} \|f\|_1 \leq M_h(R).\]  Our main result says that the two quantities in the inequality above are actually comparable.

\begin{thm} \thlabel{T:Maindualitytheorem}
    The estimate \[ \sup_{f \in \FF_n(R,h)} \|f\|_1 \simeq_n M_h(R)\] holds, with the implied constant depending only on the dimension $n$.
\end{thm}

More precisely, our proof gives the proportionality constant $3^n$, so that we have
\[\sup_{f \in \FF_n(R,h)} \|f\|_1 \leq M_h(R) \leq 3^n \cdot \sup_{f \in \FF_n(R,h)} \|f\|_1.\] Surely, the constant in this estimate is not optimal, but its precise value is not presently of interest.

\thref{T:Maindualitytheorem} can be interpreted as a functional version of the Frostman duality in \eqref{E:FrostmanDuality}. Our proof of the theorem has certain bits in common with classical proofs of Frostman's lemma (as presented, for instance, in \cite{carleson1967selected}) but we face a significant difficulty in being unable to preserve the pointwise condition $f_m \leq R$ on a sequence $\{f_m\}_m$ of partial  solutions under weak limits, whereas the corresponding support condition in the classical problem is easily handled by functional analytic tools. Our approach involves the construction of an appropriate partial solution using a primitive version of Ford-Fulkerson network flow algorithm from \cite{ford1956maximal}, and a compactness theorem for so-called \textit{bottlenecks}. The proof of \thref{T:Maindualitytheorem} is given in Section~\ref{S:DualityThmProofSection}.

\subsection{Application to weighted polynomial approximation}
\label{S:IntroSecApplALL}

We will apply \thref{T:Maindualitytheorem} to solve a problem of weighted polynomial approximation in the plane. Let $\mathbb{C}$ denote the complex plane, $z$ the independent complex variable, and \[ \Po = \Big\{ p(z) = \sum_{k=0}^n p_k z^k : p_k \in \mathbb{C}, n \in \mathbb{N}_0 \Big\}\] the set of polynomials in $z$.

\subsubsection{$\Po^t(\mu)$-spaces} Given a non-negative finite Borel measure $\mu$ with compact support in $\mathbb{C}$, and $t \in (0, \infty)$, the space $\Po^t(\mu)$ is the closure of $\Po$ in the Lebesgue space $\LL^t(\mu)$ in the natural topology. For $t = 2$, the multiplication operator $M_z: f(z) \mapsto zf(z)$ is a model for subnormal operators, in the sense that any subnormal operator satisfying some natural assumptions will be unitarily equivalent to $M_z$ on a space $\Po^2(\mu)$ for some measure $\mu$. A (biased) selection of highlights of the theory is Brown's invariant subspace existence theorem for subnormal operators in \cite{brown1978some}, Aleman-Richter-Sundberg paper \cite{aleman2009nontangential} on function and operator theory in $\Po^t(\mu)$, and Thomson's solution in \cite{thomson1991approximation} to the Mergelyan-Brennan problem, which asserts that $\Po^t(\mu) \neq \LL^t(\mu)$ if and only if $\Po^t(\mu)$ admits a so-called analytic point evaluation functional: a point $\lambda \in \mathbb{C}$ such that \[f \mapsto f(\lambda)\] is bounded in a neighbourhood of $\lambda$. 

It is known that every space $\Po^t(\mu)$ can be decomposed as a direct sum 
\begin{equation}
    \label{E:ThomsonDecompStruct}
    \Po^t(\mu) = \Big( \bigoplus_{k \geq 1} \Po^t(\mu|B_k) \Big) \oplus \LL^t(\mu|B_0)
\end{equation} where the sets $B_k$ form a Borel partition of the support of $\mu$, and $\mu|B_k$ denotes the corresponding measure restriction. Thomson's theorem from \cite{thomson1991approximation} established that to the pieces indexed by $k \geq 1$ there corresponds a simply connected domain $U_i$ of analytic point evaluations for $\Po^t(\mu|B_k)$. The \textit{splitting problem} concerns the structure of this decomposition in the special case that $\mu$ is supported in the closure of the unit disk $\D := \{ z \in \mathbb{C} : |z| < 1\}$ and has the form
\begin{equation}
    \label{E:MuPlaneEq}
    d\mu(z) = G(1-|z|) dA(z) + w(z) |dz|
\end{equation} where $G(x)$ is a non-negative function defined for $x \in (0,1]$, $w(z)$ is a non-negative Borel measurable and integrable function on $\T := \partial \D$, $dA$ is the area measure on $\D$, and $|dz|$ is the arclength measure on $\T$. Even in this special case, the structure of the decomposition \eqref{E:ThomsonDecompStruct} can be complicated. We always have
\begin{equation} \label{E:ThomsonDecompStructSpecialCase}
\Po^t(\mu) = \Po^t(\mu|\D_+) \oplus \LL^t(\mu|\T_-)    
\end{equation} where 
$\D_+ = \D \cup E$ and $\T_- = \T \setminus E$ for some Borel measurable set $E \subset \T$. The structure of the set $E$ depends on the interplay between the weight pair $(G,w)$. The exact shape of the decomposition for superexponentially decaying $G$ is given in \cite{malman2023revisiting}. 

\subsubsection{Splitting} \label{S:SplittingIntroSubsec} The phenomenon of \textit{splitting} is said to occur if we have \[ \D_+ = \D\,  \text{ and } \,  \T_- = \T\]
in the decomposition \eqref{E:ThomsonDecompStructSpecialCase}.
The nomenclature is justified by an equivalent formulation of the problem, which asks to characterize the pairs $(G,w)$ for which a sequence $\{p_k\}_k$ of polynomials in $\Po$ exists such that, simultaneously, we have
\begin{enumerate}
    \item $p_k \to 0$ in the topology of the space $\LL^t(w\, |dz|)$.
    \item $p_k \to 1$ in the topology of the space $\LL^t(G(1-|z|)dA)$,    
\end{enumerate}

The splitting problem attracted attention in the 1970s and 1980s, and below we discuss some instances of the problem that have previously been solved.

\subsubsection{Volberg's splitting theorem} In the case that $G(x)$ decays so fast as $x \to 0$ that we have 
\begin{equation} \label{E:LogLogDiv}
        \int_0^a \log \log(1/G(x)) \, dx = \infty \tag{LogLogDiv}
\end{equation}  for small $a > 0$, and $G$ satisfies some auxiliary benign regularity assumptions, then $\Po^t(\mu)$ of form \eqref{E:MuPlaneEq} splits if and only if \[ \int_\T \log w \, |dz| = -\infty.\] The statement follows by (functional analytic) duality from Volberg's theorem on summability of the logarithm of an \textit{almost analytic} function in \cite{volberg1982logarithm} and \cite{vol1987summability}, with another exposition being available in the book \cite{havinbook} by Havin and Jöricke. The works of Borichev and Volberg in \cite{borichev1990uniqueness}, and of Volberg in \cite{volberg1978mean} and \cite{volberg1988uniqueness}, present interesting results in a similar direction.

\subsubsection{Superexponential splitting theorem} In the range of $G$ for which the divergence condition \eqref{E:LogLogDiv} fails, but $G$ decays superexponentially, in the sense that 
\begin{equation} \label{E:ExpDec}
\liminf_{x \to 0} \, x \log(1/G(x)) > 0, \tag{ExpDec}
\end{equation} splitting occurs if and only if we have \[ \int_I \log w \, |dz| = -\infty\] for every arc $I \subset \T$. This condition for splitting has been conjectured by Kriete and MacCluer in their 1990 paper \cite{kriete1990mean}, and confirmed by the second author in \cite{malman2023revisiting}. Some earlier partial results which have been obtained by using a composition operator technique can be found in \cite[Chapter 9.4]{cowen1995composition}.

\subsubsection{Partial results in subexponential regime} \label{S:PartialResSubexpReg} In the case that $G$ decays so slowly that \eqref{E:ExpDec} fails, we may set \begin{equation} \label{E:hInTermsOfGeq} h(x) := x \log (1/G(x)),\end{equation} which is a function that decays to zero as $x \to 0$. For most natural examples of $G$, the function $h$ becomes a gauge function as defined earlier. 

In this regime, results of Khrushchev from 1970s characterize the pairs $(G,w)$ for which splitting occurs in the special case that $w = 1_F$ is a characteristic function of a measurable set $F \subset \T$. In his description, Khrushchev utilizes \textit{$h$-Carleson sets}. We say that a closed set $E \subset \T$ is an $h$-Carleson set if the family $\{I_k\}_k$ of maximal disjoint arcs $I_k$ complementary to $E$ in $\T$ satisfies the condition \[ \sum_k h(|I_k|) < \infty,\] where $|I_k|$ is the arclength of $I_k$. Modulo a regularity condition (see \eqref{E:Reg3} below), Khrushchev's theorem states that splitting occurs for the pair $(G, 1_F)$ if and only if $F$ contains no $h$-Carleson subsets of positive measure. Here the gauge $h$ is constructed from $G$ using \eqref{E:hInTermsOfGeq}. To solve the problem, Khrushchev explicitly constructed a sequence of polynomials with the required properties.

In the context of general $w$, the splitting problem has been discussed in 1984 in the research problem book \cite{HavinProblemBook} authored by Havin, Khrushchev and Nikolskii. Answering a question of Kriete, Volberg showed that there exists a weight $w$ which satisfies $w > 0$ almost everywhere on $\T$, for which $\Po^2(\mu)$ splits even with $G \equiv 1$. This example confirmed that the structure of the support of $w$ cannot tell the whole story, and inspired the subsequent deeper investigation in the paper \cite{kriete1990mean} seeking to generalize Khrushchev's results to the weighted context. See \cite{HavinProblemBook}, \cite{kriete1990mean}, \cite[Chapter 9.4]{cowen1995composition} and the references in those works for several partial results on this problem.

\subsubsection{Solution to the splitting problem in the subexponential regime} \label{S:IntroApplicationsSubs}

A combination of the ``blueprint" of Khrushchev from \cite{khrushchev1978problem} for constructing splitting sequences of polynomials with our duality in \thref{T:Maindualitytheorem} leads to a solution to the splitting problem in the remaining subexponential regime. To state our result in the most universal and applicable form, we shall replace norm convergence condition $(2)$ in Section \ref{S:SplittingIntroSubsec} above with the more general conditions $(ii)$ and $(iii)$ below.

\begin{thm}
    \thlabel{T:GeneralSplittingThm} Let $h$ be a gauge function, $w$ be a non-negative function in $\LL^1(|dz|)$, and $t,\epsilon \in (0,\infty)$. If we have that \[ \int_{E} \log w \, |dz| = -\infty\] for every $h$-Carleson set $E \subset \T$ of positive Lebesgue measure, then there exists a sequence of polynomials $\{p_N\}_N \subset \Po$ such that
    \begin{enumerate}[(i)]
        \item $p_N \to 0$ in the topology of $\LL^t(w\,|dz|)$,
        \item $p_N \to 1$ uniformly on compact subsets of $\D$,
        \item the bound $|p_N(z)| \leq \exp\Big( \epsilon \frac{h(1-|z|)}{1-|z|}\Big)$ holds uniformly in $N$, for all $z \in \D$.
    \end{enumerate}
\end{thm}

The proof is given in Section~\ref{S:SplittingSec}. Note that for $h$ defined by \eqref{E:hInTermsOfGeq}, part $(iii)$ above is equivalent to $|p_N(z)|^{1/\epsilon} G(1-|z|) \leq 1$. Setting $1/\epsilon = t$, we obtain the norm convergence $(2)$ in Section~\ref{S:SplittingIntroSubsec} from $(ii)$ and $(iii)$. 

The divergence condition for $\log w$ reduces to Khrushchev's condition for the set $F$ containing no $h$-Carleson sets of positive measure in the case that $w = 1_F$.

Sharpness of \thref{T:GeneralSplittingThm} has been known since 1970s and Khrushchev's work, modulo the following Dini-type regularity condition on the gauge $h$:
\begin{equation}
    \label{E:Reg3}
    \int_0^{\ell} \frac{h(x)}{x} \, dx \, \simeq \, h(\ell).
     \tag{R3}
\end{equation}
More precisely, if $\{p_N\}_N$ is a sequence of polynomials which satisfies
    \begin{enumerate}[(i')]
        \item $p_N \to 0$ in the topology of $\LL^t(w\,|dz|)$,
        \item the bound $|p_N(z)| \leq \exp\Big( \epsilon \frac{h(1-|z|)}{1-|z|}\Big)$ holds uniformly in $N$, for all $z \in \D$,
    \end{enumerate}
then the integrability of $\log w$ on any $h$-Carleson set of positive measure forces \[ p_N(z) \to 0, \quad z \in \D,\] uniformly on compact subsets of $\D$. The above statement follows from a minor modification of Khrushchev's technique in \cite[Section 3]{khrushchev1978problem}. See, for instance, the discussion in \cite[Section 9]{kriete1990mean}.

In particular, \thref{T:GeneralSplittingThm} confirms the Kriete-MacCluer conjecture from \cite[Section 9]{kriete1990mean}. Going back to the context of $\Po^t(\mu)$-spaces, we state explicitly two corollaries of \thref{T:GeneralSplittingThm} related to the two most commonly appearing Carleson-type sets. We say that a closed set $E \subset \T$ is a classical \textit{Carleson set} if \[ \sum_{k} |I_k| \log(1/|I_k|) < \infty\] where $\{I_k\}_k$ are the maximal arcs complementary to $E$. For $\alpha \in (0,1)$, the set $E$ is an \textit{$\alpha$-Carleson set} if \[ \sum_{k} |I_k|^{\alpha} < \infty.\]
An argument using \thref{T:GeneralSplittingThm} with $h$ defined as in equation \eqref{E:hInTermsOfGeq}, and Khrushchev's original results for the necessity part, yields the following results.

\begin{cor} Let $t \in (0, \infty)$, $\beta > -1$ and
\[ d\mu(z) = (1-|z|)^{\beta} dA(z) + w(z) |dz|.\]
The space $\Po^t(\mu)$ splits if and only if \[ \int_E \log w \, |dz| = -\infty \] for every Carleson set $E$ of positive arclength measure. 
\end{cor}

It follows that the weight $w$ constructed by Volberg, which was mentioned in Section \ref{S:PartialResSubexpReg} above, is such that it satisfies $w > 0$ almost everywhere, and yet $\log w$ has divergent integral over any Carleson set of positive measure.

\begin{cor} Let $\epsilon > 0$, $\alpha \in (0,1)$ and
\[ d\mu(z) = \exp(-\epsilon (1-|z|)^{\alpha -1}) dA(z) + w(z) |dz|.\]
The space $\Po^t(\mu)$ splits if and only if \[ \int_E \log w \, |dz| = -\infty \] for every $\alpha$-Carleson set $E$ of positive arclength measure. 
\end{cor}

In conjunction with Volberg's theorem from \cite{volberg1982logarithm} and the superexponential splitting theorem from \cite{malman2023revisiting}, our \thref{T:GeneralSplittingThm} essentially completes the characterization of the weight pairs $(G,w)$ for which splitting occurs. 

Lastly, we mention that splitting theorems have various applications to Fourier analysis and operator theory. Some of those applications are described in Khrushchev's fundamental work \cite{khrushchev1978problem}, and some other, more recent ones, in \cite{malman2023revisiting} and \cite{malman2023thomson}.

\section{Duality theorem}
\label{S:DualityThmProofSection}

\subsection{Dyadic system} 
\label{S:DyadicSystemSubsec} 
For convenience, we shall assume that our roof function $R$ is supported inside the unit cube $\C^n$ with a vertex at the origin,
\begin{equation}
    \label{UnitCubeDef}
    \C^n := [0,1]^n.
\end{equation}
To the unit cube $\C^n$ in \eqref{UnitCubeDef} we associate the usual dyadic tree
\[
\DD = \bigcup_{m \geq 0} \DD_m, 
\] where $\DD_m$ consists of $2^{nm}$ closed cubes of volume $2^{-nm}$ which have disjoint interiors, and for which we have 
\[ \bigcup_{d \in \DD_m} d = \C^n. \] For $m \geq 1$, every cube $d \in \DD_m$ is contained in a unique \textit{parent} cube $d' \in \DD_{m-1}.$ It is well-known (and readily established) that from every collection $\mathcal{A} \subset \DD$ we may extract a subset $\widetilde{\mathcal{A}}$ of \textit{maximal} cubes, which have pairwise disjoint interiors, and for which we have \[ \bigcup_{d \in \mathcal{A}} d = \bigcup_{d \in \widetilde{\mathcal{A}}}d.\]

\subsection{An almost-solution} We will use an iterative construction to establish the following almost-solution to our problem.

\begin{lem} \thlabel{L:mainlemmapartialsol}
    For each positive integer $m$, there exists a measurable function $f_m \in \FF_n(R, h)$ and a collection $\{d_{m,k}\}_k$ of finitely many dyadic cubes in $\bigcup_{k=0}^{m-1} \DD_k$ with disjoint interiors, for which we have
    \begin{equation} \label{E:MainLemmaProp1}
        3^n \int_{d_{m,k}} f_m \, dV_n = h(V_n(d_{m,k})),
    \end{equation}
    \begin{equation} \label{E:MainLemmaProp2} \sum_{k} h(V_n(d_{m,k})) \leq h(1)
    \end{equation} and 
    \begin{equation} \label{E:MainLemmaProp3} 3^n  \int\displaylimits_{\C^n} f_m \, dV_n = \int\displaylimits_{\C^n \setminus \cup_k d_{m,k}} \min \{ R, 2^{nm} h(2^{-nm}) \} \, dV_n + \sum_{k} h(V_n(d_{m.k})). \end{equation}
\end{lem}

For the reader familiar with the theory of combinatorial optimization, we mention that with the correct interpretation, the above lemma can be proved using the Ford-Fulkerson maximal flow theorem, initially established in \cite{ford1956maximal}. To apply the Ford-Fulkerson theorem, we would need to associate to our dyadic tree $\DD$ a graph $G$ with appropriately weighted edge set, and then construct $f_m$ from a \textit{maximal flow} running through $G$, and the set $\{d_{m,k}\}_k$ from a \textit{minimal cut}. A version of this construction appears in the proof of Frostman's lemma given by Bishop and Peres in \cite{bishop2017fractals}. This procedure is somewhat long, and we will not prove the lemma in this way. The special structure of our problem is such, that there is a shorter and direct solution. In the proof below, we shall imitate the Ford-Fulkerson maximum flow algorithm. The dyadic cubes $\{d_{m,k}\}_k$ in \thref{L:mainlemmapartialsol} will appear as bottlenecks when running the algorithm. The argument below certainly has bits in common with well-known proofs of Frostman's lemma (for instance, the one given in \cite[Chapter 8]{mattila1999geometry}), but the classical approach does not take us all the way to \thref{T:Maindualitytheorem}.

\begin{proof}[Proof of \thref{L:mainlemmapartialsol}] Let an integer $m \geq 1$ be fixed, and introduce the measurable function
\begin{equation}
    \label{E:fmInitDef}
    f_m^{\{m\}} := \min \{ R , 2^{nm} h(2^{-nm})\} \text{ on } \C^n.
\end{equation} We shall iteratively construct a sequence of functions \[ f_m^{\{m\}} \geq f_m^{\{m-1\}} \geq \ldots \geq f_m^{\{0\}}\] with the above inequalities interpreted as holding pointwise up to a set of $V_n$-measure zero. In the end, we will set $f_m$ to be a scalar multiple of $f^{\{0\}}_m$.

We have just constructed $f^{\{k\}}_m$ for $k = m$. Next we show how to construct $f^{\{k-1\}}_m$ given $f^{\{k\}}_m$, for every $k \in \{1,\ldots, m\}$. Let $f_m^{\{k\}}|d$ be the function which coincides with $f_m^{\{k\}}$ on the interior of a dyadic cube $d$ and vanishes elsewhere. Then the equality \[f_m^{\{k\}} = \sum_{ d \in \DD_{k-1}} f_m^{\{k\}}|d\] holds pointwise except for a set of $V_n$-measure zero. For $d\in \DD_{k-1}$, introduce the non-negative constants
\begin{equation} \label{E:cdConstdef}
c_d = \begin{cases}
1, & \text{ if } \int_d f^{\{m\}}_m \, dV_n \leq h(V_n(d)), \\  
\frac{h(V_n(d))}{\int_d f^{\{m\}}_m \, dV_n}, & \text{ otherwise.}
\end{cases}
\end{equation}
Note that $c_d \leq 1$ for all $d \in \DD_{k-1}$. We set
\begin{equation}
f_m^{\{k-1\}} := \sum_{d \in \DD_{k-1}} c_d f_m^{\{k\}}|d.
\end{equation} Then $f_m^{\{k\}} \geq f_m^{\{k-1\}}$. By construction, we have 
\begin{equation} \label{E:DkIneq}
\int_d f_m^{\{k-1\}} \, dV_n \leq h(V_n(d)), \quad d \in \DD_{k-1}\end{equation} and, importantly, we have equality in \eqref{E:DkIneq} if $c_d < 1$. We iterate the construction until we reach $k=0$, and next we check how far $f_m^{\{0\}}$ is from being a member of $\FF_n(R,h)$. Since \[ f_m^{\{0\}} \leq f_m^{\{m\}} \leq R\] holds almost everywhere on $\C^n$, clearly condition $(ii)$ for membership in $\FF_n(R,h)$ stated in \thref{D:FclassDef} is satisfied. We proceed to check condition $(i)$. Let $c$ be a cube, and assume first that $V_n(c) \leq 2^{-nm}$. Then we can use that $h(x)/x$ is decreasing in $x$ to estimate 
\begin{align*}
\int_c f_m^{\{0\}} \, dV_n &\leq \int_c f_m^{\{m\}}\, dV_n \\ 
&\leq V_n(c) \cdot 2^{nm}h(2^{-nm}) \\
&\leq V_n(c) \cdot V_n(c)^{-1} h(V_n(c)) \\
&= h(V_n(c)).
\end{align*} 
In the case that $V_n(c)> 2^{-nm}$, let $r$ be the unique integer smaller than or equal to $m$ such that $2^{-r} \leq V_n(c)^{1/n} \leq 2^{-r+1}$. Then $c$ is contained in the union of at most $3^n$ cubes in $\DD_{r}$. Consequently, by our construction, and in particular by \eqref{E:DkIneq}, we have
\begin{align*}
    \int_c f_m^{\{0\}} \, dV_n &\leq \int_c f_m^{\{r\}} \, dV_n \\
    &\leq 3^n h(2^{-{rn}}) \\
    &\leq 3^n h(V_n(c)).
\end{align*}
It follows that if we set \[ f_m := 3^{-n} f_m^{\{0\}},\] then $f_m \in \FF_n(R, h)$.

We now construct the collection $\{d_{m,k}\}_k$. Consider those dyadic cubes in $\bigcup_{k=1}^{m-1} \DD_k$ for which $c_d < 1$ in \eqref{E:cdConstdef}, and let $\{d_{m,k}\}_k$ be the maximal subset of those cubes, as described in Section \ref{S:DyadicSystemSubsec}. The maximality property implies that for any cube $d'$ strictly containing one of the $d_{m,k}$ we must have $c_{d'} = 1$, which by our construction means that if $d_{m,k} \in \DD_r$, then 
\begin{equation}
    \label{E:bottleNeckEq}
    3^n \int_{d_{m,k}} f_m \, dV_n = \int_{d_{m,k}} f^{\{r\}}_m \, dV_n = h(V_n(d_{m,k})),
\end{equation}
where the second equality holds since we have equality in \eqref{E:DkIneq} whenever $c_d < 1$. Hence property \eqref{E:MainLemmaProp1} holds. If $x \in \C^n$ is not contained in $\bigcup_k d_{m,k}$, then it follows that $c_d = 1$ for each of the dyadic cubes $d$ containing $x$. Thus $3^n f_m(x) = f_m^{\{m\}}(x)$ for $x \in \C^n \setminus \bigcup_k d_{m,k}$, and so property \eqref{E:MainLemmaProp3} holds as a consequence of \eqref{E:MainLemmaProp1} and \eqref{E:fmInitDef}. Finally, since by construction we have 
\[\int_{\C^n} f_m^{\{0\}}\, dV_n \leq h(V_n(\C^n)) = h(1),\] it follows from \eqref{E:DkIneq} in the case $k=1$, and from \eqref{E:bottleNeckEq}, that 
\begin{align*}
    \sum_k h(V_n(d_{m,k})) &= \sum_k 3^n\int_{d_{m,k}} f_m \, dV_n  \\
    &\leq \int_{\C^n} f_m^{\{0\}}\, dV_n \\
    &\leq h(1),
\end{align*}  which establishes \eqref{E:MainLemmaProp2}.
\end{proof}

\subsection{Bottleneck compactness}

The bottleneck sets $\{d_{m,k}\}_k$ from \thref{L:mainlemmapartialsol} obey the following compactness property, critical in the proof of our main theorem.

\begin{lem} \thlabel{L:BottleneckCompactness}
For each integer $m \geq 1$, let $\{d_{m,k}\}_k$ be a sequence of dyadic cubes in $\C^n$, where within each sequence the cubes have disjoint interiors, and for which we have 
\[ \sup_{m} \sum_k h(V_n(d_{m,k})) < \infty\] 
and
\[0 < \inf_m \sum_{k} V_n(d_{m,k}).\] Then there exists a positive number $L$, a sequence of dyadic cubes $\{d_k\}_k$ with disjoint interiors satisfying \[ \sum_k V_n(d_k) = L\] and a subsequence $\{m'\}$ of the index parameters \{m\} such that \[ \lim_{m' \to \infty} \sum_k V_n(d_{m',k}) = L\] and \[ d_{m',k} = d_k\] for all sufficiently large $m'$.
\end{lem}

In the statement above, and in the proof below, we allow for the case of degenerated empty cubes $d_{m,k} = \emptyset$.

\begin{proof}[Proof of \thref{L:BottleneckCompactness}]
    By the second assumption in the statement of the lemma and by passing to a subsequence of the index parameter $m$, we may suppose that 
    \[\lim_m \sum_{k} V_n(d_{m,k}) = L \in (0,1].\]
    Without loss of generality, we may order each of the sequences $\{d_{m,k}\}_k$ so that the volumes $V_n(d_{m,k})$ are non-increasing in $k$, that is,
    \[ V_n(d_{m, 1}) \geq V_n(d_{m, 2}) \geq \ldots \]
    We claim that \[ \liminf_m \, V_n(d_{m,1}) > 0,\] that is, the largest cube in each of the families indexed by $m$ has volume bounded away from $0$. For if not, then for any $\epsilon > 0$ we could find a large index parameter $m_0$ such that $V_n(d_{m_0,1}) < \epsilon$, from which our ordering assumption would imply that $V_n(d_{m_0,k}) < \epsilon$ for all $k \geq 1$. We would obtain
    \begin{align*}
        \sum_k h(V_n(d_{m_0,k})) &= \sum_k V_n(d_{m_0,k})\frac{h(V_n(d_{m_0,k}))}{V_n(d_{m_0,k})} \\
        &\geq \Big( \sum_k V_n(d_{m_0,k})\Big) \frac{h(\epsilon)}{\epsilon} \\
        &\geq \frac{L}{2 }\frac{h(\epsilon)}{\epsilon},
    \end{align*} where the estimate between the first and second lines is a consequence of the assumption \eqref{E:Reg1} on $h$ from Section~\ref{S:PrimalSubsec}, and the volume estimate in the last line holds if $m_0$ is sufficiently large. By \eqref{E:Reg2}, the last quantity can be made arbitrarily large by choice of $\epsilon$ sufficiently small, which would contradict our first hypothesis in the statement of the lemma.
    
    Thus $\liminf_m V_n(d_{m,1}) > 0$, and it follows that each of the cubes in the sequence $\{d_{m,1}\}_m$ is a member of the finite subset of dyadic cubes in $\DD$ of a volume bounded from below by some constant. But then by passing to a subsequence of the index parameter $m$ again, we may assume that in fact $\{d_{m,1}\}_m$ is a eventually a constant sequence: \[d_{m,1} = d_1, \quad \text{ for large } m.\]
    Let $L_1 := L - V_n(d_1).$ If $L_1 = 0$, we set $d_k = \emptyset$ for each $k \geq 2$, and terminate the algorithm. If not, then we may repeat our previous argument with sequences $\{d_{m,k}\}_{k \geq 2}$ (that is, the original sequences with first element removed) and $L_1$ in place of $L$. By passing to a subsequence we may, again, assume that the sequence $\{d_{m,2}\}_m$ is eventually constant and equal to some dyadic cube $d_2$. Repeating the argument for consecutive indices $k=2,3, \ldots$ we obtain a sequence of numbers \[ L_k := L - \sum_{s=1}^k V_n(d_s) \geq 0\] and the sequence $\{d_{m,k}\}_m$ is eventually constant, and equal to $d_k$, for each $k$.

    If $L_k > 0$ for all $k$, then to finish the proof we must now show that in fact we must have $\lim_k L_k = 0$. The argument is similar to one given above. If not, then for some positive $\alpha$ we have \[\alpha < L - \sum_{s=1}^\infty V_n(d_s).\] Hence \[ \sum_{s=1}^\infty V_n(d_s) < L - \alpha,\] and, of course, we have \[ \lim_{s \to \infty} V_n(d_s) = 0.\] Given $\epsilon > 0$, choose $s_0$ so large that $V_n(d_s) < \epsilon/2$ for all $s \geq s_0$. For all sufficiently large $m_0$, we will then also have $V_n(d_{m_0,s}) < \epsilon$ for all $s \geq s_0$. Moreover, since we have \[ \lim_m \sum_{s=1}^{s_0} V_n(d_{m,s}) = \sum_{s=1}^{s_0} V_n(d_{s}),\] we can choose $m_0$ so large that\[ \sum_{s=1}^{s_0} V_n(d_{m_0,s}) < L - \alpha,\] and \[ \sum_{s=1}^\infty V_n(d_{m_0,s}) > L - \alpha/2.\] Then \[ \sum_{s=s_0+1}^\infty V_n(d_{m_0,s}) > \alpha/2.\] But, similarly to above
    \begin{align*}
        \sum_{s=s_0+1}^\infty h(V_n(d_{m_0,s})) &= \sum_{s=s_0+1}^\infty V_n(d_{m_0,s})\frac{h(V_n(d_{m_0,s}))}{V_n(d_{m_0,s})} \\
        &\geq \Big(\sum_{s=s_0+1}^\infty V_n(d_{m_0, s})\Big) \frac{h(\epsilon)}{\epsilon} \\
        &\geq \frac{\alpha}{2}\frac{h(\epsilon)}{\epsilon}.
    \end{align*} and we arrive at a contradiction as before.
\end{proof}

\subsection{Proof of the duality theorem}

Before diving into the final proof, we shall explicitly state and prove one elementary integral convergence lemma.

\begin{lem}\thlabel{L:IntOverConv}
Let $R: \C^n \to [0,\infty]$ be measurable, $\{N_m\}_m$ be an increasing sequence of positive numbers growing to $\infty$, $\{\alpha_m\}_m$ be a sequence of measurable sets with $V_n(\alpha_m) \to 0$, and $E \subset \C^n$ be a set of positive volume measure. Then
\[ \lim_m \int_{E \setminus \alpha_m} \min(R, N_m) \, dV = \int_E R \, dV_n.\]
\end{lem}

\begin{proof} The "$\leq$" inequality is obvious. We prove the reverse. For every fixed positive number $N$, we have \[ \lim_{m} \int_{\alpha_m}\min(R, N) \, dV_n = 0 \] by absolute continuity of the finite measure $\min(R,N)dV_n$. Thus 
\begin{align*}
\int_E \min(R,N)\, dV_n &=  \lim_m \int_{E\setminus \alpha_m} \min(R,N) \, dV_n  \\
&\leq \liminf_m \int_{E\setminus \alpha_m} \min(R,N_m)\, dV_n.
\end{align*}
The result follows from the above inequality upon letting $N \to \infty$.
\end{proof}

\begin{proof}[Proof of \thref{T:Maindualitytheorem}] Let $\{f_m\}_m$ and $\B_m := \cup_{k} d_{m,k}$ be the almost-solutions and the corresponding bottleneck sets given by \thref{L:mainlemmapartialsol}.

We are done if we can prove that $M_h(R) \leq 3^n \cdot \limsup_{m} \|f_m\|_1$, where as before, $M_h(R)$ is given by \eqref{E:GenHausdorffContDef}. 

We split the proof of this inequality into three cases.

\textbf{Case 1:} we have access to a subsequence of small bottleneck sets, namely, it holds that $\liminf_m V_n(\B_m) = 0$.

In this case, we may assume that $\lim_{m} V_n(\B_m) = 0$. By applying equation \eqref{E:MainLemmaProp3} and \thref{L:IntOverConv} with $N_m = 2^{nm} h(2^{-nm})$ -- which is an increasing sequence as a consequence of \eqref{E:Reg2} -- and $\B_m = \alpha_m$, we obtain
\begin{align*}
\limsup_m \, 3^n \|f_m\|_1  &= \limsup_m \int\displaylimits_{\C^n \setminus \B_m} \min \{ R, 2^{nm} h(2^{-nm}) \} \, dV_n + \sum_{k} h(V_n(d_{m.k})) \\
&\geq \int_{\C^n} R \, dV_n \\
&\geq M_h(R).
\end{align*} 

\textbf{Case 2:} we have access to large bottleneck sets, namely, it holds that $\limsup_m V_n(\B_m) = 1$. 

In this case, the claim follows from subadditivity of the gauge $h$. Indeed, we have  
\begin{align*}
\limsup_m \, 3^n \|f_m\|_1  &\geq \limsup_m \sum_{k} h(V_n(d_{m.k})) \\
&\geq \limsup_m h \left( \sum_k V_n(d_{m,k} )\right) \\
&\geq h(1).
\end{align*}
Since, by our initial assumption, $\C^n$ covers the support of $R$, we have \[h(1) = h(V_n(\C^n)) \geq M_h(R).\]

\textbf{Case 3:} the bottleneck volumes $V_n(\B_m)$ are uniformly bounded away from $0$ and $1$.

In this case we need to use our bottleneck compactness result from \thref{L:BottleneckCompactness}. We apply it to the sequences $\{d_{m,k}\}_k$ and obtain a corresponding sequence $\{d_k\}_k$ of cubes. Let $E := \C^n \setminus \bigcup_k d_k$. After passing to a subsequence, we may assume that $d_{m,k} = d_k$ for all sufficiently large $m$, and that $V_n(\B_m) \to L \in (0,1)$. The sets \[ E_m := E \setminus \bigcup_k d_{m,k}\] are asymptotically not much smaller than $E$. More precisely, the set $E_m$ is obtained from $E$ by removing a tail $T_m := \bigcup_{k \geq k(m)} d_{m,k}$ from $E$, where $k(m)$ is the largest integer such that $d_{m,k} = d_k$ for all $k < k(m)$. Since for each integer $N$, we may choose $m$ such that $d_{m',k} = d_k$ for all $k<N$ and $m' \geq m$, the cut-off integer $k(m) \to \infty$ as $m \rightarrow \infty$. 


Now, since \[ \lim_{m} \sum_k V_n(d_{m,k}) = \sum_k V_n(d_k) = L\] and \[\lim_m V_n(d_{m,k}) = V_n(d_k), \text{ for all } k,\] an elementary convergence argument shows that \[ \limsup_m \sum_{k \geq k(m)} V_n(d_{m,k}) = 0,\] and hence the tail volumes $V_n(T_m)$ tend to zero as $m \to \infty$. Note also that by Fatou's lemma, we have \[ \liminf_m \sum_k h(V_n(d_{m,k})) \geq \sum_k h(V_n(d_k)).\] Combining the above with \thref{L:IntOverConv} for $\alpha_m = T_m$ and $N_m = 2^{nm}h(2^{-nm})$, we obtain

\begin{align*}
\limsup_m \, 3^n \|f_m\|_1  &= \limsup_m \int\displaylimits_{\C^n \setminus \B_m} \min \{ R, 2^{nm} h(2^{-nm}) \} \, dV_n + \sum_{k} h(V_n(d_{m.k})) \\
&\geq \limsup_m \int\displaylimits_{E \setminus T_m} \min \{ R, 2^{nm} h(2^{-nm}) \} \, dV_n + \sum_{k} h(V_n(d_k)) \\
&= \int_E R \, dV_n + \sum_{k} h(V_n(d_k)) \\
&= \int_{\C^n \setminus \cup_k d_k} R \, dV_n + \sum_{k} h(V_n(d_k)) \\
&\geq M_h(R).
\end{align*}
We have thus completed the proof of the duality theorem.
\end{proof}

\section{Splitting sequences}
\label{S:SplittingSec}

We now indicate how one may combine the duality theorem with Khrushchev's method from \cite{khrushchev1978problem} to prove \thref{T:GeneralSplittingThm}. Parts of our argument will be somewhat sparse on details, since the only novelty here is the way in which \thref{T:Maindualitytheorem} enters the picture through \thref{L:DualityBuildingBlock}. For this reason, instead of repeating the details, we shall at multiple points refer to the proof of Khrushchev's result in \cite{havinbook} and \cite{khrushchev1978problem}.

Let $\hil^\infty(\D)$ be the algebra of bounded analytic functions in the unit disk $\D$. Every function $H \in \hil^\infty(\D)$ has an extension to the unit circle $\T$ through radial limits which exist $|d \zeta|$-almost everywhere, where $|d \zeta|$ denotes the usual arclength (Lebesgue) measure on $\T$. If we construct a sequence $\{H_N\}_N$ in $\hil^\infty(\D)$ which satisfies the conditions in \thref{T:GeneralSplittingThm}, then those conditions hold also for the sequence of polynomials $\{p_N\}_N$ which are appropriately truncated Abel means of $H_N$. We construct $H_N$ as the so-called \textit{outer functions}
\begin{equation}
    \label{E:hNeq}
    H_N(z) := \exp\Big( \frac{\gamma_N}{2\pi} \int_\T \frac{\zeta + z}{\zeta - z} f_N(\zeta) |d \zeta| \Big), \quad z \in \D,
\end{equation} for appropriately chosen sequence of positive numbers $\{\gamma_N\}_N$ tending to $0$, and bounded real-valued functions $f_N$ defined on $\T$. It is well-known that the equality $|H_N(\zeta)| = \exp (\gamma_N f_N(\zeta))$ holds $|d \zeta|$-almost everywhere on $\T$. 
Khrushchev constructs his sequence $\{H_N\}_N$ in \cite[Section 4.2]{khrushchev1978problem} from functions $f_N$ of the form 
\begin{equation}
    \label{E:fNstruct}
    f_N = \sum_k f_{N,k}
\end{equation}
where $f_{N,k}$ live on disjoint arcs on $\T$ and have appropriate properties. We follow the same strategy. For $E \subset \T$, let $|E|$ denote its Lebesgue measure on $\T$. The following lemma is the only new result of this section.

\begin{lem} \thlabel{L:DualityBuildingBlock} Let $w$ be as in \thref{T:GeneralSplittingThm}, and set $R := \log^+(1/w)$. Let furthermore $N$ be a large positive integer, $I$ be an arc of $\T$, and \[A_N := \{ \zeta \in \T : R(\zeta) = \log^+(1/w(\zeta)) \leq N \}\] a sublevel set of $R$. Assume that $|A_N \cap I| > 0$. There exists a numerical constant $C_0 > 0$ and a bounded real-valued function $f_{N,I}$ supported on $I$ which satisfies 
\begin{enumerate}[(i)]
    \item $f_{N,I} \leq R$ almost everywhere on $\T$ with respect to $|d \zeta|$,
    \item $\int_\Delta f_{N,I} \, |d \zeta| \leq h(|\Delta|)$ for any arc $\Delta \subset \T$,
    \item $\int_I f_{N,I} \, |d \zeta| = 0$,
    \item $f_{N,I}(\zeta) \leq -C_0 \frac{h(|I|)}{|I|}$ for $\zeta \in A_N \cap I$.
\end{enumerate}    
\end{lem}

\begin{proof}
    This follows from an application of \thref{T:Maindualitytheorem} in the case $n=1$. Let $R_{N,I}$ be the function coinciding with $R$ on the set $I \cap \{ \zeta \in \T : R(\zeta) > N\} = I \setminus A_N$ and which vanishes elsewhere on $\T$. If $\{I_k\}_k$ is any sequence of arcs of $\T$, then the quantity 
    \begin{equation} \label{E:MhQuantityCircle}
        \sum_{k} h(|I_k|) + \int_{\T \setminus \cup_k I_k} R_{N,I}\, |d \zeta|
    \end{equation} is finite, by assumption, only if $|I \setminus \cup_k I_k| = 0$. Indeed, otherwise the hypothesis forces the second term to be infinite, since if $I_k^o$ denotes the interior of $I_k$, then $E := I \setminus \cup_k I_k^o$ is easily seen to be an $h$-Carleson set of positive arclength measure (possibly after having added the endpoints of $I$, in the event that $I$ was not a closed arc), on which the integral of $R$, and hence of $R_{N,I}$, diverges. It follows that $\sum_k |I_k| \geq |I|$, and so subadditivity of $h$ implies that \eqref{E:MhQuantityCircle} is bounded below by $h(|I|)$. Now an application of \thref{T:Maindualitytheorem} to $R_{N,I}$ produces a non-negative function $f_{N,I}^+$ (we may suppose that $f^+_{N,I}$ is bounded by an appropriate truncation) which satisfies part $(i)$ and $(ii)$ above, and \[\int_I f_{N,I}^+ \, |d \zeta| \geq C_o h(|I|)\] for some independent constant $C_0 > 0$. Note that $f_{N,I}^+$ vanishes on $A_N \cap I$, since $R_{N,I}$ does. We set \[f_{N,I} = f^+_{N, I} - p_{N,I} 1_{A_N \cap I}\] where $1_{A_N \cap I}$ is the characteristic function of the set $A_N \cap I$ and \[ p_{N,I} = \frac{\int_I f^+_{N, I} \,|d \zeta|}{|A_N \cap I|}\] is a positive constant. Clearly $(i)$ and $(ii)$ hold for $f_{N,I}$. The choice of $p_{N,I}$ is such that $(iii)$ holds also. Moreover, for $\zeta \in A_N \cap I$, we have \[ -f_{N,I}(\zeta) = p_{N,I} \geq C_0\frac{h(|I|)}{|A_N \cap I|} \geq C_0\frac{h(|I|)}{|I|}\] which proves $(iv)$.
\end{proof}

\begin{proof}[Proof of \thref{T:GeneralSplittingThm}] Fix an integer $N$ and let $A_N$ be as in \thref{L:DualityBuildingBlock}. Let $\{I_{N,k}\}_k$ be $N$  disjoint arcs of $\T$ (half-open, say), each of length $2\pi/N$, so that they together form a partition of $\T$. We construct $f_N$ as in \eqref{E:fNstruct}, where $f_{N,k} = f_{N,I_{N,k}}$ is as in \thref{L:DualityBuildingBlock} if $|A_N \cap I_{N,k}| > 0$, and $f_{N,k} = 0$ otherwise. By \eqref{E:Reg2} and part $(iv)$ of \thref{L:DualityBuildingBlock}, there exists a sequence of positive numbers $\{\gamma_N\}_N$ tending to $0$ for which we have that 
\begin{equation}
    \label{E:fNCollapseToMinusInfty}
    \lim_{N \to \infty} \gamma_N f_N(\zeta) = -\infty 
\end{equation}$|d \zeta|$-almost everywhere on \[\bigcup_N A_N = \{ \zeta \in \T : R(\zeta) < \infty \} = \{ \zeta \in \T : w(\zeta) > 0 \}.\] The bound \[ \int_\Delta f_N \, |d \zeta| \leq 2 h(|\Delta|)\] for any arc $\Delta \subset \T$ follows from parts $(ii)$ and $(iii)$ of \thref{L:DualityBuildingBlock} as in Khrushchev's original proof (see \cite[Section 4.3]{khrushchev1978problem} or \cite[p. 300]{havinbook}). By a classical Poisson integral estimate, this leads to the bound \[ |H_N(z)| \leq \exp \Big( 2 \gamma_N C_1 \frac{h(1-|z|)}{1-|z|} \Big)\] (see \cite[Lemma 4.2]{khrushchev1978problem} or \cite[p. 297]{havinbook} for details). Here $C_1 > 0$ is some independent constant, and $H_N$ is as in \eqref{E:hNeq}. By part $(iii)$ of \thref{L:DualityBuildingBlock} and \eqref{E:hNeq}, we have $H_N(0) = 1$, and (since $\gamma_N \to 0$) the above inequality shows that $\limsup_N |H_N(z)| \leq 1$ for all $z \in \D$. By the maximum modulus principle we see that properties $(ii)$ and $(iii)$ in \thref{T:GeneralSplittingThm} hold for $\{H_N\}_N$. It remains to verify $(i)$. Note that we have \[f_N(\zeta) \leq R(\zeta) = \log^+(1/(w(\zeta))\] by part $(i)$ of \thref{L:DualityBuildingBlock}, and so as soon as $\gamma_N$ is small enough for the inequality $t\gamma_N \leq 1$ to hold, on $\T$ we have the inequality 
\begin{align*}
    |H_N|^t w &= \exp ( t \gamma_N f_N) w  \\
    & \leq \exp( \log^+(1/w))w \\
    &\leq 1 + w.
\end{align*} Finally, combining the equality $|H_N| = \exp(\gamma_N f_N)$ on $\T$, the estimate \eqref{E:fNCollapseToMinusInfty}, and the Dominated Convergence Theorem shows that $(i)$ of \thref{T:GeneralSplittingThm} also holds. Thus the proof of \thref{T:GeneralSplittingThm} is complete. 
\end{proof}

\bibliographystyle{plain}
\bibliography{mybib}

\end{document}